\documentclass[a4paper, 12pt,reqno]{amsart}
\usepackage{amsmath,amssymb,amsthm,enumerate}
\allowdisplaybreaks
\theoremstyle{plain}
\newtheorem{theorem}{Theorem}
\newtheorem*{theorem*}{Theorem}
\newtheorem{lemma}{Lemma}
\newtheorem*{lemma*}{Lemma}
\theoremstyle{definition}

\newtheorem*{definition*}{Definition}
\theoremstyle{remark}
\newtheorem{remark}{Remark}
\newtheorem*{remark*}{Remark}
\theoremstyle{example}

\newtheorem*{example*}{Example}
\theoremstyle{problem}

\newtheorem*{problem*}{Problem}
\theoremstyle{hypothesis}

\newtheorem*{hypothesis*}{Hypothesis}
\theoremstyle{corollary}
\newtheorem{corollary}{Corollary}
\newtheorem*{corollary*}{Corollary}

\begin{document}
\title[Singular representations of real numbers]{One example of singular representations of real numbers from the unit interval}

\author{Symon Serbenyuk}
\maketitle
\text{\emph{simon6@ukr.net}}\\
\text{\emph{Kharkiv National University of Internal Affairs, Ukraine }}

\begin{abstract}

In this article,  for modelling numeral systems, the operator approach, which is introduced in \cite{Symon2023},  is generalized for a certain case. An example of such numeral systems is introduced and  considered. 
 \maketitle

\end{abstract}


\section{Introduction}

In the modern science, it is widely used modelling of various numeral systems and their investigations for constructing objects and techniques of applied mathematics for modelling real objects and phenomena in economics, physics, computer science and computer security, social sciences, etc. One can note that constructing of new representations of real numbers   is an important tool for the following areas of science: applications of non-cracking the coding mechanism in cybersecurity, encoding and decoding information, and modelling and studying of  ``pathological" mathematical objects (Cantor and Moran sets, singular functions, non-differentiable, or nowhere monotonic functions, etc.; the notion of  ``pathology" in mathematics is explained in \cite{35}). Pathological mathematical objects are widely applied in various areas of science (see \cite{1, 5, 11, 18, 36, 37, 38, Symon2023}).

Finally, let us remark that there exist a number of research devoted to modeling and investigations of different numeral systems and pathological mathematical objects (for example, see surveys in \cite{1, 5, IS2009, Renyi1957, 10, 15, 76, 70,  S21}, etc.).

In the present research,  certain results from the paper \cite{Symon2023} are generalized, as well as investigations of  problems introduced in the last-mentioned paper, are begun. 

An aproach for modelling new representations of real numbers is based on using certain type functions by several ways of constructing. One can use pathological functions for modelling expansions of real numbers and vice versa, as well as one can note the other technique  for modelling Cantor series expansions with non-integer (fractional) bases (see \cite{Symon2023} and the first preprint of the last paper \cite{Symon-preprint2019} of 2019), quasy-nega-representations, etc. (se also \cite{S.Serbenyuk}).

Expansions introduced in this paper can be modeled by the singular Salem function and a certain operator of changes of digits in the representation. Let us begin with definitions. 

Let $q>1$ be a fixed positive integer. It is well known that any number $x\in [0,1]$ can be represented by the following form
$$
\sum^{\infty} _{k=1}{\frac{i_k}{q^k}}:= \Delta^{ q} _{i_1i_2...i_k...}=x,
$$
where $i_k\in A \equiv\{0,1,\dots , q-1\}$. The last-mentioned representation is called \emph{the q-ary representation} of $x$. 

Let $\eta$ be a random variable defined by a $q$-ary expansion 
$$
\eta= \frac{\xi_1}{q}+\frac{\xi_2}{q^2}+\frac{\xi_3}{q^3}+\dots+\frac{\xi_{k}}{q^{k}}+\dots := \Delta^{q} _{\xi_1\xi_2...\xi_{k}...},
$$
where digits $\xi_k$ $(k=1,2,3, \dots)$ are random and taking the values $0,1,\dots ,q-1$ with probabilities ${p}_{0}, {p}_{1}, \dots , {p}_{q-1}$. That is, $\xi_k$ are independent and  $P\{\xi_k=i_k\}={p}_{i_k}$, $i_k \in A$. Here $0<p_t<1$, $P_q=\{p_0, p_1, \dots , p_{q-1}\}$, and $\sum^q _{t=0}{p_t}=1$.

From the definition of distribution function and the  expressions 
$$
\{\eta<x\}=\{\xi_1<i_1\}\cup\{\xi_1=i_1,\xi_2<i_2\}\cup\cdots 
$$
$$
\cdots \cup\{\xi_1=i_1,\xi_2=i_2,\dots ,\xi_{k-1}<i_{k-1}\}\cup
$$
$$
\cup\{\xi_1=i_1,\xi_2=i_2,\dots ,\xi_{k-1}=i_{k-1},\xi_{k}<i_{k}\}\cup \dots,
$$
$$
P\{\xi_1=i_1,\xi_2=i_2,\dots ,\xi_{k-1}=i_{k-1}, \xi_k<\varepsilon_{k}\}=
\beta_{i_{k}}\prod^{k-1} _{r=1} {{p}_{i_r}},
$$
it follows that the following is true: the distribution function  ${S}_{\eta}$ of the random variable $\eta$ can be represented in the following form for $x=\Delta^q _{i_1i_2...i_k...}\in [0, 1]$:
$$
{S}_{\eta}(x)=S(x)=\beta_{i_1}+\sum^{\infty} _{k=2} {\left({\beta}_{i_k} \prod^{k-1} _{r=1} {{p}_{i_r}}\right)}.
$$
The last-mentioned function was introduced by Salem in \cite{Salem} and is called \emph{the Salem function}. This function is an increasing singular function.

Let us consider an analytic representation of the Salem function as an expansion of real numbers from $[0, 1]$, since this function is a defined and continuous on $[0, 1]$. That is,
$$
{S}(x)=\beta_{i_1}+\sum^{\infty} _{k=2} {\left({\beta}_{i_k} \prod^{k-1} _{r=1} {{p}_{i_r}}\right)}=\Delta^{P_q} _{i_1i_2...i_k...}=x\in [0, 1].
$$

The notation $\Delta^{P_q} _{i_1i_2...i_k...}$ is called \emph{ the singular Salem representation  (or $P_q$-representation) of $x\in [0, 1]$}.

\begin{remark}
It is well-known that for a $q$-ary expansions, numbers of the form
$$
\Delta^{q} _{i_1i_2...i_{m-1}i_m00000...}=\Delta^{q} _{i_1i_2...i_{m-1}i_m(0)}=\Delta^{q} _{i_1i_2...i_{m-1}[i_m-1][q-1][q-1][q-1]...}=\Delta^{q} _{i_1i_2...i_{m-1}[i_m-1](q-1)}
$$
are called \emph{$q$-rational}. The rest of numbers are \emph{$q$-irrational} and have the unique $q$-representation. 

Since $S$ is a distribution function, we get that numbers of the form
$$
\Delta^{P_q} _{i_1i_2...i_{m-1}i_m(0)}=\Delta^{P_q} _{i_1i_2...i_{m-1}[i_m-1](q-1)}
$$
are \emph{$P_q$-rational}. The rest of numbers are \emph{$P_q$-irrational} and have the unique $P_q$-representation. 
\end{remark}

Let us consider the case when $q=3$. Let $\theta$ be an operator defined on digits $\{0, 1, 2\}$ by the following rule: $\theta(0)=0$, $\theta(1)=2$, and  $\theta(2)=1$.

In this paper, the main attention is given to the $P_{\theta}$-representation of real numbers which related to the $P_3$-representation by the following rule:
\begin{equation}
\label{eq:1}
\Delta^{P_{\theta}} _{\gamma_1\gamma_2..\gamma_k...}=\Delta^{P_3} _{\theta(i_1)\theta(i_2)...\theta(i_k)...}:=\beta_{\theta(i_1)}+\sum^{\infty} _{k=2}{\left(\beta_{\theta(i_{k})}\prod^{k-1} _{r=1} {p_{\theta(i_r)}}\right)},
\end{equation}
where $i_k \in\{0, 1, 2\}$.

Numbers of the form
$$
\Delta^{P_\theta} _{\gamma_1\gamma_2...\gamma_{m-1}\gamma_m(0)}=\Delta^{P_\theta} _{\gamma_1\gamma_2...\gamma_{m-1}[\gamma_m-1](2)}
$$
are \emph{$P_\theta$-rational}. The rest of numbers are \emph{$P_\theta$-irrational} and have the unique $P_\theta$-representation. 

The $P_\theta$-representation of real numbers is a generalization of the $3^{'}$-representation (\cite{Symon2023}) and coincides with the last representation under the condition $p_0=p_1=p_2=\frac 1 3$.

To investigate relationships of certain generalized positive and alternating (sign-variable)  expansions of real numbers, let us consider an operator aproach and model a simple example of numeral system, the geometry of which is a generalization of the geometries of some positive and alternating expansions.

\section{The basis of the metric theory}

\begin{lemma}
Each number $x\in[0,1]$ can be represented in terms of the $P_\theta$-representation \eqref{eq:1}. In addition, every  $P_\theta$-irrational number has the unique $P_\theta$-representation, and every $P_\theta$-rational number has two different representations.
\end{lemma}
\begin{proof}
Let us consider the following function
\begin{equation}
\label{eq: 2}
f: x=\Delta^{P_3} _{i_1i_2...i_k...}~~~\to~~~\Delta^{P_3} _{\theta(i_1)\theta(i_2)...\theta(i_k)...}=\Delta^{P_\theta} _{\gamma_1\gamma_2...\gamma_k...}=f(x)=y,
\end{equation}
where $\theta(0)=0$, $\theta(1)=2$, and $\theta(2)=1$.

The last function was investigated in \cite{Symon23}, its partial cases are considered in \cite{{S. Serbenyuk preprint2}, {S. Serbenyuk functions with complicated local structure 2013}}. Let us recall some its properties:
\begin{itemize}
\item this function is continuous at $P_3$-irrational points, and $P_3$-rational points are points of its discontinuity;
\item the function is a singular nowhere monotone function;
\item the graph is a fractal dust.
\end{itemize}

The statement follows from the last-mentioned properties of $f$.
\end{proof}

Suppose $c_1,c_2,\dots , c_m$ is an ordered tuple of digits from $\{0,1,2\}$. Then \emph{a cylinder $\Lambda^{P_\theta} _{c_1c_2...c_m}$ of rank $m$ with base $c_1c_2\ldots c_m$} is a set of the form:
$$
\Lambda^{P_\theta} _{c_1c_2...c_m}\equiv\left\{x: x=\Delta^{P_\theta} _{c_1c_2...c_m\gamma_{m+1}\gamma_{m+2}\gamma_{m+3}...}, \gamma_t\in\{0,1,2\}, t>m \right\},
$$

\begin{lemma}
\label{lm: cylinder's properties}
Cylinders $\Delta^{P_\theta} _{c_1c_2...c_m}$ have the following properties:
\begin{enumerate}
\item Any cylinder $\Lambda^{P_\theta} _{c_1c_2...c_m}$ is a closed interval, as well as 
$$
\Lambda^{P_\theta} _{c_1c_2...c_m}:=\left[\Delta^{P_\theta} _{c_1c_2...c_m(0)},\Delta^{P_\theta} _{c_1c_2...c_m(2)}\right].
$$
\item For the Lebesgue measure $|\cdot |$ of a set, the following holds:
$$
\left|\Lambda^{P_\theta} _{c_1c_2...c_m}\right|=\prod^{m} _{r=1}{p_{c_r}}=\prod^{m} _{r=1}{p_{\theta(d_r)}},
$$
where $\theta(d_r)=c_r$ for all $r=\overline{1, m}$.
\item
$$
\Lambda^{P_\theta} _{c_1c_2...c_mc}\subset \Lambda^{P_\theta} _{c_1c_2...c_m}.
$$
\item
$$
\Lambda^{P_\theta} _{c_1c_2...c_m}=\bigcup^2 _{c=0}{\Lambda^{P_\theta} _{c_1c_2...c_mc}}.
$$
\item Cylinders $\Lambda^{P_\theta} _{c_1c_2...c_m}$ are  left-to-right situated.
\end{enumerate}
\end{lemma}
\begin{proof} Choose a certain $x_0\in \Lambda^{P_3} _{d_1d_2...d_m}$ such that $\theta(d_r)=c_r$ for all $r=\overline{1, m}$.

For proving, one can use properties of $f$. Then
\begin{equation}
\label{eq: 18}
\begin{split}
f(x_0)&=f\left(\Delta^{P_3} _{d_1d_2...d_mi_{m+1}i_{m+2}...}\right)\\
&=\Delta^{P_3} _{\theta(d_1)\theta(d_2)...\theta(d_m)\theta(i_{m+1})\theta(i_{m+2})...}\\
&=\Delta^{P_\theta} _{c_1c_2...c_m\gamma_{m+1}\gamma_{m+2}...}\in \Lambda^{P_\theta} _{c_1c_2...c_m}.
\end{split}
\end{equation}
Also,
\begin{equation}
\label{eq: 19}
f\left(\inf \Lambda^{P_3} _{d_1d_2...d_m}\right)=f\left(\Delta^{P_3} _{d_1d_2...d_m(0)}\right)={\Delta^{P_\theta} _{c_1c_2...c_m(0)}}=\inf\Lambda^{P_\theta} _{c_1c_2...c_m}\in \Lambda^{P_\theta} _{c_1c_2...c_m},
\end{equation}
\begin{equation}
\label{eq:20}
f\left(\sup \Lambda^{P_3} _{d_1d_2...d_m}\right)=f\left(\Delta^{P_3} _{d_1d_2...d_m(2)}\right)={\Delta^{P_\theta} _{c_1c_2...c_m(1)}}\in \Lambda^{P_\theta} _{c_1c_2...c_m},
\end{equation}
and
\begin{equation}
\label{eq: 21}
\sup\Lambda^{P_\theta} _{c_1c_2...c_m}=\Delta^{P_\theta} _{c_1c_2...c_m(2)}=f\left(\Delta^{P_3} _{d_1d_2...d_m(1)}\right).
\end{equation}

Let us prove that a cylinder is a segment. Suppose that $x\in \Delta^{P_\theta} _{c_1c_2...c_m}$. That is,
$$
x=\beta_{c_1}+\sum^{m} _{k=2}{\left(\beta_{c_k}\prod^{k-1} _{r=1}{p_{c_r}}\right)}+\left(\prod^{k} _{r=1}{p_{c_r}}\right)\left(\beta_{\gamma_{m+1}}+\sum^{\infty} _{t=m+2}{\left(\beta_{\gamma_t}\prod^{t-1} _{s=m+1}{p_{\gamma_s}}\right)}\right), 
$$
where $\gamma_{t}\in\{0,1,2\}$ for $t=m+1, m+2, m+3, \dots$.
Whence,
$$
x^{'}=\beta_{c_1}+\sum^{m} _{k=2}{\left(\beta_{c_k}\prod^{k-1} _{r=1}{p_{c_r}}\right)}\le x\le \beta_{c_1}+\sum^{m} _{k=2}{\left(\beta_{c_k}\prod^{k-1} _{r=1}{p_{c_r}}\right)}+\left(\prod^{k} _{r=1}{p_{c_r}}\right)=x^{''}.
$$
So $x\in\left[x^{'}, x^{''}\right]\supseteq\Lambda^{P_\theta} _{c_1c_2...c_m}$. Since
$$
x^{'}=\beta_{c_1}+\sum^{m} _{k=2}{\left(\beta_{c_k}\prod^{k-1} _{r=1}{p_{c_r}}\right)}+\left(\prod^{k} _{r=1}{p_{c_r}}\right)\inf\left(\beta_{\gamma_{m+1}}+\sum^{\infty} _{t=m+2}{\left(\beta_{\gamma_t}\prod^{t-1} _{s=m+1}{p_{\gamma_s}}\right)}\right)
$$
and
$$
x^{''}=\beta_{c_1}+\sum^{m} _{k=2}{\left(\beta_{c_k}\prod^{k-1} _{r=1}{p_{c_r}}\right)}+\left(\prod^{k} _{r=1}{p_{c_r}}\right)\sup\left(\beta_{\gamma_{m+1}}+\sum^{\infty} _{t=m+2}{\left(\beta_{\gamma_t}\prod^{t-1} _{s=m+1}{p_{\gamma_s}}\right)}\right)
$$
hold for any $x\in \Delta^{P_\theta} _{c_1c_2...c_m}$, we obtain $x, x^{'}, x^{''}\in  \Delta^{P_\theta} _{c_1c_2...c_m}$.

So, \emph{the first and second properties are proven}.

Let us prove \emph{the third property}. Suppose ${d_1, d_2 \dots , d_m}$ is a fixed tuple of digits from $\{0, 1, 2\}$ such that $\theta(d_r)=c_r$ for all $r=\overline{1, m}$. Then
$$
f\left(\inf \Lambda^{P_3} _{d_1d_2...d_m d}\right)=\inf\Lambda^{P_\theta} _{c_1c_2...c_nc}\in \Delta^{P_\theta} _{c_1c_2...c_nc},
$$
where $\theta(d)=c$, and
$$
\inf\Lambda^{P_\theta} _{c_1c_2...c_nc}< f\left(\sup \Lambda^{P_3} _{d_1d_2...d_m d}\right) <\sup\Lambda^{P_\theta} _{c_1c_2...c_nc}.
$$
Hence, using the first and second properties of cylinders, as well as relationships \eqref{eq: 18}-\eqref{eq: 21}, we have
$$
\inf{\Lambda^{P_\theta} _{c_1c_2...c_mc}}\ge \inf{\Lambda^{P_\theta} _{c_1c_2...c_m}}
$$
and
$$
\sup{\Lambda^{P_\theta} _{c_1c_2...c_mc}}\le \sup{\Lambda^{P_\theta} _{c_1c_2...c_m}}.
$$

So, 
$$
\Delta^{P_\theta} _{c_1c_2...c_mc}\subset \Delta^{P_\theta} _{c_1c_2...c_m}=\bigcup^2 _{c=0}{\Delta^{P_\theta} _{c_1c_2...c_mc}}.
$$

Let us prove tha last property. Let us consider the differences:
$$
\inf\Lambda^{P_\theta} _{c_1c_2...c_{m-1}2}-\sup\Lambda^{P_\theta} _{c_1c_2...c_{m-1}1}=\Delta^{P_\theta} _{c_1c_2...c_{m-1}2(0)}-\Delta^{P_\theta} _{c_1c_2...c_{m-1}1(2)}=0,
$$
$$
\inf\Lambda^{P_\theta} _{c_1c_2...c_{m-1}1}-\sup\Lambda^{P_\theta} _{c_1c_2...c_{m-1}0}=\Delta^{P_\theta} _{c_1c_2...c_{m-1}1(0)}-\Delta^{P_\theta} _{c_1c_2...c_{m-1}0(2)}=0.
$$
Our Lemma is proven. 
\end{proof}

\begin{theorem} The map
$$
f: x=\Delta^{P_3} _{i_1i_2...i_k...} \to\Delta^{P_3} _{\theta(i_1)\theta(i_2)...\theta(i_k)...}=\Delta^{P_\theta} _{\gamma_1\gamma_2...\gamma_k...}=y
$$
does not preserve a distance  between points and the Lebesgue measure of an  interval (segment).
\end{theorem}
\begin{proof}
Let us prove that $f$ does not preserve a distance. Let us choose $x_1, x_2\in[0,1]$ such that  the condition $|f(x_2)-f(x_1)|\ne |x_2-x_1|$ holds. The statement follows from the existence of jump discontinuities of $f$. Really, for example, suppose that $p_0=\frac 1 2$, $p_1=\frac 1 3$, $p_2=\frac 1 6$, as well as $x_1=\Delta^{P_3} _{22(0)}$ and $x_2=\Delta^{P_3} _{21(0)}$. Then 
$$
|x_1-x_2|=\beta_2+\beta_2p_2-\beta_2-\beta_1p_2=p_2(\beta_2-\beta_1)=p_2(p_0+p_1-p_0)=p_1p_2=\frac{1}{18}
$$
and
$$
|f(x_2)-f(x_1)|=\left|\Delta^{P_\theta} _{11(0)}-\Delta^{P_\theta} _{12(0)}\right|=|\beta_1+\beta_1p_1-\beta_1-\beta_2p_1|=|p_1(\beta_1-\beta_2)|=\frac 1 9\ne|x_1-x_2|.
$$

Let us prove that $f$ does not preserve the Lebesgue measure. Suppose  $[x_1, x_2]\subset [0,1]$  is a segment. 

If $[x_1, x_2]\equiv \Lambda^{P_3} _{d_1d_2...d_m}$, then considering the last lemma, we get 
$$
\left|\Lambda^{P_3} _{d_1d_2...d_m}\right|=\prod^{m} _{t=1}{p_{d_t}}
$$
and
$$
\left|f\left(\Lambda^{P_3} _{d_1d_2...d_m}\right)\right|=\left|\Lambda^{P_\theta} _{c_1c_2...c_m}\right|=\prod^{m} _{t=1}{p_{c_t}}=\prod^{m} _{t=1}{p_{\theta(d_t)}}.
$$
In  a general case, 
$$
\prod^{m} _{t=1}{p_{d_t}}\ne\prod^{m} _{t=1}{p_{\theta(d_t)}}.
$$
Really, suppose $p_0=\frac 1 4$, $p_1=\frac 1 2$, and $p_2=\frac 1 4$. For example, then
$$
\left|\Lambda^{P_3} _{11122}\right|=p^3 _1 p^2 _2=\frac 1 8\cdot \frac{1}{16}=\frac{1}{128}
$$
but
$$
\left|f\left(\Lambda^{P_\theta} _{22211}\right)\right|=p^3 _2 p^2 _1=\frac{1}{64}\cdot\frac 1 4=\frac{1}{256}.
$$

Let $[x_1, x_2]\subset [0,1]$ be a segment that are not a certain cylinder $\Delta^{P_3} _{d_1d_2...d_m}$; then there exists $\varepsilon$-covering by cylinders of rank $k$ and $m$ such that 
$$
\bigcup_m \Delta^{P_3} _{d_1d_2...d_m}\subseteq
[x_1, x_2]\subseteq\bigcup_k \Delta^{P_3} _{d_1d_2...d_k}
$$
and 
$$
\lim_{k\to\infty}{\left|\bigcup_k \Delta^{P_3} _{d_1d_2...d_k}\right|}=\lim_{m\to\infty}{\left|\bigcup_m \Delta^{P_3} _{d_1d_2...d_m}\right|}=|[x_1, x_2]|.
$$
Whence, in this case, our results depnds on the case when $[x_1, x_2]$  is a certain cylinder.
\end{proof}

\begin{remark}
It is useful the problem on geometry of a positive expansion of real numbers and its corresponding alternating or sign-variable expansion. This connection is useful for studying metric, dimensional and other properties of mathematical objects defined by these representations of real numbers. The last lemma give the answer on this problem.

Let $\Delta_{i_1i_2...i_k...}, i_k\in A_k,$ be the representation of a number $x$ from some interval by a positive expansion. One can model the corresponding alternating  (sign-variable; here the case of alternating expansions is explained) representation  $\Delta^{-}_{j_1j_2...j_k...}$ by the following way (relationship):
$$
\Delta^{-}_{j_1j_2...j_k...}\equiv \Delta_{\theta^{'}(i_1)i_2\theta^{'}(i_3)i_4...\theta^{'}(i_{2k-1}).i_{2k}...},
$$
where cylinders of even rank are left-to-right situated and cylinders of odd rank are right-to-left situated, as well as
$$
\theta^{'}(i_{2k-1})=\max_{i_{2k-1}\in A_{2k-1}}{\{i_{2k-1}\}}- i_{2k-1},
$$
or
$$
\Delta^{-}_{j_1j_2...j_k...}\equiv \Delta_{i_1\theta^{'}(i_2)i_3\theta^{'}(i_4)i_5...i_{2k-1}\theta^{'}(i_{2k})...}
$$
and here cylinders of odd rank are left-to-right situated but cylinders of even rank are right-to-left situated, as well as
$$
\theta^{'}(i_{2k})=\max_{i_{2k}\in A_{2k}}{\{i_{2k}\}}- i_{2k}.
$$

Let us consider an example. For the case of $P_3$-representation, we obtain
$$
\Delta^{-P_3}_{j_1j_2...j_k...}\equiv \Delta^{P_3} _{i_1[2-i_2]i_3[2-i_4]i_5...i_{2k-1}[2-i_{2k}]...}
$$
or
$$
\Delta^{-P_3}_{j_1j_2...j_k...}\equiv \Delta^{P_3} _{[2-i_1]i_2[2-i_3]i_4...i_{2k-2}[2-i_{2k-1}]...}.
$$
Suppose even rank cylinders are left-to-right situated, as well as $p_0=\frac 1 2, p_1=\frac 1 3$, and $p_2=\frac 1 6$; then
$$
\left|\Lambda^{P_3} _{121200}\right|=p^2 _0p^2 _1p^2 _2=\frac{1}{4\cdot 9\cdot 36}=\frac{1}{1296}
$$
and
$$
\left|\Lambda^{-P_3} _{121220}\right|=p_0p^2 _1p^3 _2=\frac{1}{2\cdot 9\cdot 216}=\frac{1}{3888}.
$$
\end{remark}
\begin{corollary}
In a general case, the metric theories of the $P_q$- and quasi-nega-$P_q$-representations can be different. 
\end{corollary}


\begin{thebibliography}{9}


\bibitem{1}	P. Billingsley. Probability and Measure (2nd ed.), Wiley (New York, 1995).


\bibitem{Cantor1} G.~Cantor, Ueber die einfachen Zahlensysteme.  
{Z. Math. Phys.},    {14}: 121--128, 1869.

\bibitem{5}	K. Falconer. Fractal Geometry: Mathematical Foundations and Applications (2nd ed.), John Wiley \& Sons Ltd, Chichester, 2003. 

\bibitem{76} J. Galambos.  Representations of Real Numbers by Infinite Series. Lecture Notes in Mathematics, vol. 502. Springer (1976)


\bibitem{IS2009} S.~Ito and T.~Sadahiro, Beta-expansions with negative bases. {Integers}, {9}: 239--259, 2009.

\bibitem{10}	S. Kalpazidou, A. Knopfmacher, J. Knopfmacher. L\"uroth-type alternating series representations for real numbers, \emph{Acta Arithmetica}  55 (1990),  311--322. 


\bibitem{11}	H. Katsuura. Continuous Nowhere-Differentiable Functions - an Application of Contraction Mappings,\emph{The American Mathematical Monthly} 98 (1991), no. 5, 411-416.     


\bibitem{15}	J. L\"uroth. Ueber eine eindeutige Entwickelung von Zahlen in eine unendliche Reihe, \emph{Math. Ann.} 21 (1883), 411--423. 


\bibitem{70} J. Neunh\"auserer.  Non-uniform Expansions of Real Numbers, \emph{ Mediterr. J. Math.} 18, Article 70 (2021). doi:10.1007/s00009-021-01723-7

\bibitem{MathPubl2022} J. Neunh\"auserer, Representations of Real Numbers Induced by Probability Distributions on $\mathbb N$, \emph{Tatra Mountains Mathematical Publications}, vol.82, no.2, 2022, pp.1-8. https://doi.org/10.2478/tmmp-2022-0014

\bibitem{Renyi1957} A.~R\'enyi, Representations for real numbers and their ergodic properties. 
{Acta. Math. Acad. Sci. Hungar.}, 8:  477--493, 1957.

\bibitem{18}	O. E. Rossler, C. Knudsen,  J. L. Hudson, I. Tsuda. Nowhere-differentiable attractors, \emph{ International Journal for Intelligent Systems} 10 (1995), no. 1, 15--23. 

\bibitem{Salem} R. Salem. On some singular monotonic functions which are stricly increasing. {Trans.
Amer. Math. Soc.}, 53:  423--439, 1943.



\bibitem{Schweiger2016} F. Schweiger. {Continued Fractions and Their Generalizations: A Short History of $f$-Expansions}. 
Boston: Docent Press, Massachusetts, 2016.


\bibitem{Schweiger(184)2018} Fritz Schweiger. Invariant measures for Moebius maps with three
branches.  {J. Number Theory},  {184}: 206--215, 2018.

\bibitem{S. Serbenyuk functions with complicated local structure 2013}{S. Serbenyuk}.
 On one class of functions with complicated local structure. {\v{S}iauliai Mathematical Seminar}, {11 (19)}: 75--88, 2016.


\bibitem{Serbenyuk2016}
S.~ Serbenyuk. Nega-$\tilde Q$-representation as a generalization of certain alternating representations
of real numbers. {Bull. Taras Shevchenko Natl. Univ. Kyiv Math. Mech.}, 1 (35),  32-39, 2016, Ukrainian. 
Available at https://www.researchgate.net/publication/308273000 

\bibitem{Serbenyuk2017}
S. Serbenyuk. Representation of real numbers by the alternating Cantor series. { Integers}, {17}:  Paper No. A15, 2017, 27 pp.


\bibitem{Symon2017} S. O. Serbenyuk. Continuous Functions with Complicated Local Structure Defined in Terms of Alternating Cantor Series Representation of Numbers. {Zh. Mat. Fiz. Anal. Geom.}, 13 (1):   57--81, 2017. https://doi.org/10.15407/mag13.01.057
 


\bibitem{S. Serbenyuk preprint2} {S. Serbenyuk}. Non-differentiable functions defined in terms of classical representations of real numbers.  {Zh. Mat. Fiz. Anal. Geom.}, 14 (2):  197--213, 2018. https://doi.org/10.15407/mag14.02.197


\bibitem{S.Serbenyuk} \emph{S. Serbenyuk } On some generalizations of real numbers representations, arXiv:1602.07929v1 (in Ukrainian)

\bibitem{Symon-preprint2019} Symon Serbenyuk. On one approach to modeling numeral systems, arXiv:1907.10534v1, 2019, 11 pages.


\bibitem{S.Serbenyuk 2018}{S. Serbenyuk}. Generalizations of certain representations of real numbers, \emph{Tatra Mountains Mathematical Publications} \textbf{77} (2020), 59--72. https://doi.org/10.2478/tmmp-2020-0033, arXiv:1801.10540


\bibitem{S21} Symon Serbenyuk, Systems of functional equations and generalizations of certain functions. \emph{Aequat. Math.} 95, 801--820 (2021). doi: 10.1007/s00010-021-00840-8

\bibitem{Symon2023}  Symon Serbenyuk. Some types of numeral systems and their modeling. \emph{The Journal of Analysis } {\bf {31}}, 149–177 (2023). https://doi.org/10.1007/s41478-022-00436-8

\bibitem{Symon23} Symon Serbenyuk.One modification of the Salem function, arXiv:2304.07776 , 9 pages.
  
\bibitem{35}	Wikipedia contributors, ``Pathological (mathematics)", Wikipedia, The Free Encyclopedia, https://en.wikipedia.org/wiki/Pathological\_(mathematics) 

\bibitem{36}	Wikipedia contributors, ``Fractal", Wikipedia, The Free Encyclopedia,  https://en.wikipedia.org/wiki/Fractal   

\bibitem{37}	Wikipedia contributors, ``Thomae's function", Wikipedia, The Free Encyclopedia,  https://en.wikipedia.org/wiki/Thomae's\_function   

\bibitem{38}	Wikipedia contributors, ``Singular function", Wikipedia, The Free Encyclopedia,  https://en.wikipedia.org/wiki/Singular\_function     





\end{thebibliography}
\end{document}